\documentclass[a4paper, reqno]{amsart}

\pdfoutput=1

\usepackage{amsmath,amsthm,amssymb}
\usepackage[alphabetic]{amsrefs}
\usepackage{mathtools}
\usepackage{lmodern}
\usepackage{microtype}
\usepackage[T1]{fontenc}
\usepackage{dsfont}
\usepackage[english]{babel}
\usepackage[hidelinks]{hyperref}

\allowdisplaybreaks

\newtheorem{theorem}{Theorem}[section]
\newtheorem{lemma}[theorem]{Lemma}
\newtheorem{prop}[theorem]{Proposition}

\newtheorem{conjecture}[theorem]{Conjecture}
\theoremstyle{definition}
\newtheorem{remark}[theorem]{Remark}

\numberwithin{equation}{section}

\DeclareMathOperator{\Br}{Br}
\DeclareMathOperator{\Cone}{Cone}

\DeclareMathOperator{\Ample}{Ample}
\DeclareMathOperator{\Frac}{Frac}
\DeclareMathOperator{\Pic}{Pic}
\DeclareMathOperator{\Eff}{Eff}
\DeclareMathOperator*{\im}{im}
\DeclareMathOperator*{\Spec}{Spec}
\DeclareMathOperator*{\Proj}{Proj}
\DeclareMathOperator{\rk}{rk}

\newcommand{\Cd}{\mathds{C}}
\newcommand{\Qd}{\mathds{Q}}
\newcommand{\Qbar}{\overline{\Qd}}
\newcommand{\Zd}{\mathds{Z}}
\newcommand{\Fd}{\mathds{F}}
\newcommand{\Rd}{\mathds{R}}
\newcommand{\Gd}{\mathds{G}}
\newcommand{\Pd}{\mathds{P}}
\newcommand{\Ad}{\mathds{A}}
\newcommand{\Gda}{\mathds{G}_{\mathrm{a}}}
\newcommand{\Gdm}{\mathds{G}_{\mathrm{m}}}
\newcommand{\SL}{\mathrm{SL}}
\newcommand{\Lc}{\mathcal{L}}
\newcommand{\Oc}{\mathcal{O}}
\newcommand{\Cc}{\mathcal{C}}
\newcommand{\Xc}{\mathcal{X}}
\newcommand{\Dm}{\mathfrak{D}}
\newcommand{\Um}{\mathfrak{U}}
\newcommand{\Xm}{\mathfrak{X}}
\newcommand{\Ym}{\mathfrak{Y}}
\newcommand{\Zm}{\mathfrak{Z}}
\newcommand{\df}{\mathrm{d}}
\newcommand{\an}{\mathrm{an}}
\newcommand{\et}{\textrm{ét}}
\newcommand{\prim}{\mathrm{prim}}
\newcommand{\sums}[1]{\sum_{\substack{#1}}}
\newcommand{\ints}[1]{\int_{\substack{#1}}}
\newcommand{\norm}[1]{\lVert #1 \rVert}

\AtBeginDocument{%
   \def\MR#1{}
}

\begin{document}

\title[Iitaka fibrations and integral points]
{Iitaka fibrations and integral points:\\ a family of arbitrarily polarized \\
spherical threefolds}

\author{Ulrich Derenthal}

\address{Institut f\"ur Algebra, Zahlentheorie und Diskrete Mathematik, Leibniz
  Universit\"at Hannover, Welfengarten 1, 30167 Hannover, Germany}
\email{derenthal@math.uni-hannover.de}

\author{Florian Wilsch}

\address{Mathematisches Institut, Georg-August-Universität Göttingen, Bunsenstr.\ 3-5, 37073 Göttingen, Germany}
\email{florian.wilsch@mathematik.uni-goettingen.de}

\date{May 15, 2025}

\begin{abstract}
  Studying Manin's program for a family of spherical log Fano threefolds, we determine the asymptotic number of integral points whose height associated with an arbitrary ample line bundle is bounded. This confirms a recent conjecture by Santens and sheds new light on the logarithmic analogue of Iitaka fibrations, which have not yet been adequately formulated.
\end{abstract}

\maketitle

\microtypesetup{protrusion=false}
\tableofcontents
\microtypesetup{protrusion=true}

\vspace{0pt plus 2pt}

\section{Introduction}

How often is a $(2 \times 2)$-determinant a perfect power?
A first step in treating such a Diophantine problem is to study geometric properties of the varieties it describes. In this case, solutions 
correspond to points on the family
\begin{equation*}
  X'_n = \left\{\det\begin{pmatrix}
      a & b\\c & d
  \end{pmatrix}=z^{n+1}\right\} \subset \Pd_\Qd(1,n,1,n,1)
\end{equation*}
of subvarieties of a weighted projective space. These are Fano varieties, which tend to have many rational points.
For $n\ge 2$, they are singular, and desingularizations $\pi_n \colon X_n \to X'_n$ are obtained by blowing up the singular locus $V(a,c,z)$. Moreover, both the varieties $X'_n$ and $X_n$ are \emph{spherical}, admitting an action by the algebraic group $G=\SL_2 \times \Gdm$ (Section~\ref{sec:boundaries}).

Manin's program \cite{MR89m:11060} asks a precise quantitative question: how many rational points of bounded anticanonical height does such a variety contain? For the singular family $X'_n$, this problem was solved by the first author and Gagliardi~\cite[Thm.~1.2]{DG18}.
As $\pi_n$ is not crepant if $n\ge 3$, their interpretation of the asymptotic formula rests on a theory by Batyrev and Tschinkel~\cite{BT98polarized}, who provide a framework to analyze asymptotic formulas for the number of rational points of bounded height with respect to height functions induced by arbitrary ample (or at least big and nef) line bundles $L$.
In particular, the leading constant of their formula involves a sum over the fibers of the \emph{Iitaka fibration} associated with the bundle $\pi_n^*\omega_{X'_n}^\vee$ on $X_n$.

For integral points, a program similar to that kick-started by Manin is underway: see \cites{MR2275615,CLT10,wilsch-toric} and the references therein for contributions and work of Santens~\cite[\S\,6]{Santens23} for a current conjecture.
Notably, there is not yet a logarithmic analogue to the Iitaka fibrations appearing in the work of Batyrev and Tschinkel~\cite{BT98polarized} (originally dubbed $\Lc$-primitive fibrations by them).
The aim of this article is to start filling in this gap by studying the integral analogue of Manin's problem with respect to arbitrary polarizations on open log Fano subvarieties of $X_n$. The equation defining $X_n$ defines the model
\begin{equation*}
    \Xm'_n = \{ad-bc=z^{n+1}\} \subset \Pd_\Zd(1,n,1,n,1).
\end{equation*}
A model $\pi_n \colon \Xm_n \to \Xm_n'$ of the desingularization can be obtained by blowing up $V_{\Xm}(a,c,z)$.

Of particular interest are those subvarieties that are still spherical. There are two $G$-invariant effective divisors $D$ with strict normal crossings on each $X_n$ such that $(X_n, D)$ is log Fano, that is, such that the log anticanonical bundle $\omega_{X_n}(D)^\vee$ is ample. As the variety $X_n$ itself is not Fano, the empty divisor is not among these two. Theorem~\ref{thm:main-thm-abstract} provides an asymptotic formula for the number of integral points $x$ outside any such $D$ such that a height function $H_L$ associated with an arbitrary ample $\Qd$-divisor $L$ is bounded by a parameter $B$ tending to infinity. 
This result confirms the conjecture formulated by Santens~\cite{Santens23}; in those cases in which he conjectures one, this includes his leading constant.

Some cases (those that are not \emph{adjoint rigid}) fall beyond the scope of his conjecture in terms of the leading constant. We observe that the fibrations appearing in these cases are completely analogous to those for rational points and believe this to be a general phenomenon as long as the analytic Clemens complex is a simplex, that is, as long as all geometric components of the boundary divisor share a real point: in this case, the variants of the Picard group pertinent to the study of integral points coincide with the usual one~\citelist{\cite{Santens23}*{Def.~3.6, Lem.~3.10}\cite{wilsch-toric}*{Def.~2.2.1, Rmk.~2.2.2}}.

\subsection{Iitaka fibrations in the integral Manin problem}

Let $\Xc = (\Xm,D,L)$ be a \emph{polarized log Fano variety} consisting of a flat proper $\Zd$-scheme $\Xm$ whose generic fiber $X = \Xm_\Qd$ is smooth and projective, a divisor $D$ on $X$ with geometrically strict normal crossings such that the log anticanonical class $-K_X-D$ is ample, and an ample metrized line bundle $L$ on $X$. Let $H_L\colon X(\Qd)\to \Rd_{>0}$ be the height function induced by $L$.

In this setting, Santens \cite[\S\,6]{Santens23} defines constants $a(\Xc)$ and $b(\Xc)$ and conjectures that, under suitable hypotheses, there is a thin set $Z \subset X(\Qd)$ such that the number
\begin{equation}\label{eq:counting_function}
    N_{\Xc,Z}(B) := \#\{x \in \Um(\Zd) \setminus Z \mid H_L(x) \le B\}
\end{equation}
of integral points on $\Um = \Xm \setminus \overline{D}$ of bounded height grows asymptotically like
\begin{equation}\label{eq:conjecture-rigid}
    c(\Xc)B^{a(\Xc)}(\log B)^{b(\Xc)-1}
\end{equation}
for some positive constant $c(\Xc)$ as $B\to \infty$.

If all geometric components of $D$ are defined over $\Rd$ and share a real point, these constants admit the following description: let
\begin{equation*}
    a(\Xc) = \inf\{t\in \Rd \mid K_X + D + tL \in \Eff(X)\}
\end{equation*}
be the inverse of the Kodaira energy, 
\begin{equation*}
  E_\Xc = K_X+D+a(\Xc)L \in \Pic(X)_\Qd
\end{equation*}
the adjoint bundle and 
$b(\Xc)$ be the codimension of the minimal face of the effective cone containing $E_\Xc$. Under these assumptions, the polarized log variety $\Xc$ is said to be adjoint rigid if the adjoint bundle $E_\Xc$
is rigid, that is, if $H^0(X,\Oc_X(rE_{\Xc}))$ is one-dimensional for all $r\in \Qd_{>0}$ such that $rE_{\Xc}$ is integral.
In this case, Santens provides a conjecture for the leading constant $c(\Xc)$ in~\eqref{eq:conjecture-rigid} that is similar to the leading constant in Manin's conjecture by Peyre \cite{Peyre95Duke} and Batyrev--Tschinkel \cite{BT98polarized}.
Whenever this is not the case, we believe that the leading constant $c(\Xc)$ can be described in terms of the standard Iitaka fibration
\[
  \phi_{(X,L)} \colon X \dashrightarrow W=\Proj\left(\bigoplus_{k = 0}^\infty H^0(X,\Oc_X(kdE_\Xc))\right),
\]
where $d$ be a positive integer such that $dE_\Xc \in \Pic(X)$ --- the target space is a point if and only if $E_{\Xc}$ is rigid.

\begin{conjecture}\label{conj:iitaka}
  Assume that the set $\Um(\Zd)$ of integral points is not thin.
  Let $\phi=\phi_{(X,L)}$ be the Iitaka fibration associated with $L$,
  induced by its adjoint bundle $E_{\Xc}$.
  For each $t\in W(\Qd)$ for which the fiber $\phi^{-1}(t)$ is defined, let $X_t$ and $\Xm_t$ be its closures in $X$ and $\Xm$, respectively.
  
  Then there are thin sets $Z\subset X(\Qd)$ and $Y\subset W(\Qd)$ such that for every $t\in W(\Qd)\setminus Z$, the fiber $\Xc_t = (\Xm_t, D\cap X_t, L|_{X_t})$ is adjoint rigid with $a(\Xc_t)=a(\Xc)$ and $b(\Xc_t) = b(\Xc)$, such that the sum
  \begin{equation*}
      c(\Xc) = \sum_{t\in W(\Qd)\setminus Y} c(\Xc_t)
  \end{equation*}
  converges, and such that
  \begin{equation*}
    N_{\Xc,Z}(B) \sim c(\Xc) B^{a(\Xc)}(\log B)^{b(\Xc)-1}.
  \end{equation*}
\end{conjecture}

\begin{remark}
  Under these assumptions, the real component of the fibration Santens obtains while counting integral points on toric varieties~\cite[Eq.~(35)]{Santens23} coincides with the construction by Batyrev and Tschinkel~\cite[p.~332]{BT98polarized} after replacing the canonical by the log canonical bundle in the latter. Therefore, it coincides with the Iitaka fibrations discussed above.
\end{remark}

\subsection{Integral points on the spherical family}

Turning back to the spherical family $X_n$ for $n\ge 2$ and the polarized log Fano varieties $\Xc = (\Xm_n,D,L)$ with a $G$-invariant divisor $D$ and an ample line bundle $L$ that is metrized as in Section~\ref{sec:height-functions}, the Iitaka
fibration $\phi : X_n \to \Pd^k_\Qd$ is in fact a morphism, with $k\in \{0,2\}$, depending on whether the adjoint bundle is rigid or not (Section~\ref{sec:polarizations}). For each $t \in \Pd^k(\Qd)$, let $X_{n,t} = \phi^{-1}(t)$ be the fiber.

By constructions of Chambert-Loir--Tschinkel~\cite{CLT10} and Santens~\cite{Santens23}, for each fiber $\Xc_t$, the metric on $L$ induces a Tamagawa measure $\tau_{\Xc}$ on $X_{n,t}(\Ad_\Qd)$ --- a product of measures concentrated on $\Um(\Zd_p)\cap X_{n,t}(\Qd_p)$ with convergence factors for finite primes $p$ and a \emph{residue measure} concentrated on a subset of $X_{n,t}(\Rd)$ (Section~\ref{sec:leading-constants}). Finally, let $\alpha(\Xc_t)$ be the $\alpha$-constant associated with each fiber (Sections~\ref{sec:alpha}--\ref{sec:moving-constants}).

\begin{theorem}\label{thm:main-thm-abstract}
  Let $\Xc$ be one of the polarized spherical log Fano threefolds described above. The number of integral points of bounded height satisfies
  \begin{equation*}
    N_{\Xc,\emptyset}(B) = c(\Xc)
    B^{a(\Xc)}(\log B)^{b(\Xc)-1}(1+O_\Xc(1/\log B)),
  \end{equation*}
  for $B\ge 2$ with
  \begin{equation*}
    c(\Xc) = \sum_{t\in \Pd^k(\Qd)}
    \frac{\alpha(\Xc_t)}{a(\Xc)}
    \tau_{\Xc_t}(X_{n,t}(\Ad_\Qd)).
  \end{equation*}
\end{theorem}

In fact, we obtain a power-saving error term (as in Proposition~\ref{prop:rigid_integrals}) whenever $b(\Xc)=1$ (that is, whenever the adjoint bundle $E_\Xc$ is nonzero). In particular, this confirms Conjecture~\ref{conj:iitaka} for this family. The exponents and leading constants are completely explicit. The former are described in~(\ref{eq:a-invariant-Dw},\,\ref{eq:a-invariant-DwDz},\,\ref{eq:b-invariant}). In the \emph{adjoint rigid} case, the fibration $\phi$ is in fact a constant morphism. The leading constant is associated with its unique fiber: the $\alpha$-constant is described in (\ref{eq:alpha_loganticanonical}, \ref{eq:alpha_not_loganticanonical}). The Tamagawa volume decomposes as
\[
\tau_{\Xc}(X_{n,t}(\Ad_\Qd)) = \prod_v \omega_{\Xc, v};
\]
its archimedean factor is described in~(\ref{eq:omega_infty_Dw},\,\ref{eq:omega_infty_DwDz_1},\,\ref{eq:omega_infty_DwDz_2}) and contains the height function~\eqref{eq:height_torsor} its denominator, and its $p$-adic factors are described in~(\ref{eq:omega_p}). 
Whenever the adjoint bundle is \emph{moving}, the Iitaka fibration is described in~\eqref{eq:iitaka}. The leading constant associated with each fiber is described in~\eqref{eq:constants-moving-explicit}.

The ample line bundle $L$ (the polarization of $X$) associated with the height function is arbitrary. Such general height functions have mostly been studied using an approach that analyzes the height zeta function, exploiting an almost transitive action of an algebraic group; for integral points, this notably includes work by Takloo-Bighash, Tschinkel~\cite{MR3117310}, and Chow~\cite{Chow} on wonderful compactifications and by Santens~\cite{Santens23} on toric varieties (using harmonic analysis on their universal torsors). For varieties without such a special structure, the torsor method has previously only been used to count rational or integral points of bounded (log) anticanonical height; see \cites{WilschIMRN,DW24,OrtmannPaper} for applications in the integral setting.

\subsection{Acknowledgements}

We would like to thank Christian Bernert and Tim Browning for their remarks on an earlier version of this article. The first author was supported by the Deutsche Forschungsgemeinschaft (DFG) -- 512730679 (RTG2965). The second author was supported by the Deutsche Forschungsgemeinschaft (DFG) -- 398436923 (RTG2491).

\section{A family of spherical threefolds}

We start by collecting some facts on the family $X_n$. Throughout, let $n\ge 2$ be an integer. The Picard group of the desingularization $\pi_n \colon X_n \to X'_n$ is $\Pic X_n \cong \Zd^2$, generated by the strict transform of the coordinate plane $V(z)$ and the exceptional divisor. 

\subsection{Cox rings and universal torsors}\label{sec:cox}
The Cox ring of $X_n$ and combinatorial data associated with it have been computed by the first author and Gagliardi~\cite[\S 2]{DG18} using constructions by Brion~\cite{bri07} and Gagliardi~\cite{gag14}; it is
\begin{equation*}
  R(X_n) = \Qd[a,b,c,d,z,w]/(ad-bc-z^{n+1}w),
\end{equation*}
and its grading is given by
\begin{equation*}
  \deg(a)=\deg(c)=\deg(z)=(1,0),\ \deg(b)=\deg(d)=(n,1),\ \text{and}\ \deg(w)=(0,1)
\end{equation*}
so that the relation has degree $(n+1,1)$. Hence, the anticanonical class is
\begin{equation*}
  -K_{X_n} = (n+2,2)
\end{equation*}
by the adjunction formula. As
\[
    \pi_n^*(-K_{X'_n}) = (n+2,1+2/n) = \frac{n+2}{n}(n,1),
\]
the resolution $\pi_n$ is not crepant if $n\ge 3$.

To construct a universal torsor, $X_n$ can be represented as a hypersurface in the toric variety $Z_n$ that is obtained by blowing up the singular locus of $Z_n' = \Pd_\Qd(1,n,1,n,1)$. The fan of $Z_n'$ consists of all cones spanned by at most four of the vectors
\begin{equation*}
    u_a=e_1,\ u_b=e_2,\ u_c=e_3,\ u_d=e_4,\ \text{and}\ u_z=-e_1-ne_2-e_3-ne_4
\end{equation*}
in $\Zd^4 \cong \Zd^5/\Zd(1,n,1,n,1)$. The singular locus corresponds to the cone spanned by $u_a,u_c,u_z$, and subdividing it by adding the ray generated by $u_w=(u_a+u_c+u_z)/n$ yields the fan $\Sigma_n$ of $Z_n$. One can check that the maximal cones of $Z_n$ are generated by any two of $u_a,u_c,u_z$ together with any two of $u_b,u_d,u_w$. Hence, the irrelevant ideal is generated by the products of one of $a,c,z$ and one of $b,d,w$.

The spherical variety $X_n$ is the hypersurface defined by $ad-bc-z^{n+1}w$ in $Z_n$, and its integral model $\Xm_n$ is described by the same equation in the toric scheme $\Zm_n$ over $\Zd$ described by the same fan \citelist{\cite{MR1679841}*{p.~187} \cite{MR284446}}, the blow-up of $\Pd_\Zd(1,n,1,n,1)$ in $V(a,c,z)$. Salberger \cite[Remark~8.6b]{MR1679841} constructs universal torsors $\Ad^6 \setminus (V(a,c,z)\cup V(b,d,w))$ over $\Qd$ and $\Zd$ of this toric variety and scheme. Restricting these to the hypersurfaces $X_n$ and $\Xm_n$ results in the universal torsor 
\[
  \rho \colon \Spec R(X_n) \setminus (V(a,c,z)\cup V(b,d,w))= Y_n \to X_n
\]
and its model $\rho\colon \Ym_n \to \Xm_n$ with
\[
  \Ym_n = 
  \Spec(\Zd[a,b,c,d,z,w]/(ad-bc-z^{n+1}w)) \setminus (V(a,c,z)\cup V(b,d,w)).
\]

As in \cite{DG18}, rational points $x$ on $X_n$ can thus be represented in \emph{Cox coordinates}, that is, as an equivalence class $x=(a:b:c:d:z:w)$ of rational numbers $a$, $b$, $c$, $d$, $z$, and $w$ satisfying
\begin{equation*}
    (a,c,z),(b,d,w)\ne (0,0,0)
\end{equation*}
and $ad-bc=z^{n+1}w$ that share an orbit under the action of $\Gdm^2(\Qd) = (\Qd^\times)^2$ encoded in the grading of the Cox ring.
Alternatively, and more conveniently for counting purposes, they are represented by \emph{integers} that, in addition to the equation, satisfy the coprimality condition 
\begin{equation}\label{eq:gcd}
  \gcd(a,c,z)=\gcd(b,d,w)=1,
\end{equation}
with their equivalence being induced by the action of the subgroup $\Gdm^2(\Zd) = \{\pm 1\}^2$. The former representation is that of a point on $X_n$ by an orbit in $Y_n(\Qd)$, the latter by an orbit in $\Ym_n(\Zd)$.

\subsection{Admissible boundaries}\label{sec:boundaries}
Our next aim is to identify $G$-invariant divisors $D$ with strict normal crossings such that the log anticanonical bundle is ample. Recall from~\cite{DG18} that the action of $G = \SL_2 \times \Gdm$ on $X_n'$ is defined by
\begin{equation*}
    \left(\begin{pmatrix}a'&b'\\c'&d'\end{pmatrix},t\right)\cdot 
    \left(\begin{pmatrix}a&b\\c&d\end{pmatrix},z\right)=
    \left(
      \begin{pmatrix}a'&b'\\c'&d'\end{pmatrix}\cdot \begin{pmatrix}a&b\\c&d\end{pmatrix} \cdot \begin{pmatrix}t&0\\0&t^{-1}\end{pmatrix},z
    \right).
\end{equation*}
Observe that the only $G$-invariant prime divisors are $D_z=V(z)$ and $D_w=V(w)$; neither of the classes $-K_{X_n}=(n+2,2)$ and $-K_{X_n}-[D_z] = (n+1,2)$ are ample (or even nef if $n\ge 3$), while both $-K_{X_n}-[D_w]$ and $-K_{X_n}- [D_w]-[D_z] $ are.

Hence, $D$ is $G$-invariant and $-K_X-D$ ample if and only if
\begin{equation}\label{eq:G-invariant-D}
    D\in \{D_w, D_w+D_z\}.
\end{equation}
For a fixed such divisor, denote by $U = X_n \setminus D$ its complement, by $\Dm = \overline{D}$ its closure in $\Xm_n$, and by $\Um = \Xm_n \setminus \Dm$ the induced model.

\begin{remark}\label{rmk:equivariant}
  The varieties $X'_n$, $X_n$, and $X_n \setminus D_w$ are (partial) equivariant compactifications of $\Gda^3$. Indeed, a $\Gda^3$-action on $X_n'$ is given by
  \begin{equation*}
      (t_1,t_2,t_3)\cdot(a:b:c:d:z) = (a':b':c':d':z')
  \end{equation*}
  where $a'=a$, $b'=b+t_1a^n$, $c'=c+t_2a$, $z'=z+t_3a$, and
  \begin{equation*}
      d'=d+t_1t_2a^n+t_1a^{n-1}c+t_2b+\sum_{i=0}^n\tbinom{n+1}{i}t_3^{n-i+1}a^{n-i}z^i, 
  \end{equation*}
  which is $(z'^{n+1}+b'c')/a'$ if $a\ne 0$. It fixes the singular locus $V(a,c,z) \cong \Pd^1_\Qd$ and extends to the desingularization $X_n$ by letting $\Gda^3$ act trivially on the second factor of the exceptional divisor $D_w$, which turns out to be isomorphic to $V(a,c,z) \times \Pd^1_\Qd$.
  Hence, for the boundary divisor $D=D_w$, the asymptotic number of integral points on $U$ of bounded log anticanonical height is known by work of Chambert-Loir and Tschinkel~\cite{CLT12};
  however, the case of general ample height functions remained open so far.
  
  The varieties $X_n \setminus (D_w \cup D_z)$, on the other hand, have no such structure as a partial equivariant compactification of a vector group. Indeed, if one of them were a compactificaton of $\Ad^3$, the inclusion $\Ad^3\to X_n\setminus (D_w\cup D_z)$ of the open orbit would spread out to one over sufficiently large primes, leading to at least $p^3$ points modulo $p$ away from the divisors for such primes, more than the actual number $p^3-p$ of such points obtained in~\eqref{eq:points-away-from-Dw-and-Dz}.
\end{remark}

\begin{remark}
  The divisors $V(a)$ and $V(b)$ are not invariant under $G$, and hence their complement is not spherical.
  As they are invariant under a Borel subgroup, they still seem to be natural candidates for components of $D$.
  Nevertheless, no combinations of $V(a)$, $V(b)$, and possibly $D_z$ and $D_w$ that have strict normal crossings result in an ample log anticanonical class.
  On the singular variety $X_n'$, the empty divisor results in an ample (log) anticanonical class.
  As the Picard group of $X_n'$ is one-dimensional, all height functions are equivalent to a power of the anticanonical one; hence, this case is completely treated by \cite{CLT12} and \cite{DG18}.
  The only $G$-invariant prime divisor on $X'_n$ is $V_{X'_n}(z)$, whose preimage is $D_w+D_z$, a divisor that is included in our analysis.
\end{remark}

\subsection{Polarizations}\label{sec:polarizations}

The nef cone of $X_n$ is spanned by the degrees $(n,1)$ and $(1,0)$ of generators of the Cox ring that can easily be verified to correspond to base point free divisors. In particular, the ample cone is its interior
\begin{equation*}
  \Ample(X_n) = \Cone((n,1), (0,1))^\circ.
\end{equation*}
Hence, any ample $\Qd$-divisor $L$ has the shape
\begin{equation*}
    L = L_{(l_1,l_2)} = (l_1 +nl_2, l_2)\quad\text{for $(l_1,l_2)\in \Qd_{>0}^2$.}
\end{equation*}
Let $\Xc = (X_n, D, L)$ be a triple consisting of one of the spherical varieties $X_n$ with $n\ge 2$, one of the $G$-invariant divisors $D$ obtained in~\eqref{eq:G-invariant-D}, and an ample line bundle $L$ as above.
The opposite 
\[
  a(\Xc) = \inf\{t\in \Rd \mid K_{X_n} + D + tL \in \Eff(X_n)\}
\]
of the Kodaira energy is readily computed as
\begin{equation}\label{eq:a-invariant-Dw}
  a(\Xc) = \max\left\{\frac{1}{l_2}, \frac{n+2}{l_1+nl_2}\right\} = \begin{cases}
      \frac{1}{l_2} & \text{if }l_1 > 2l_2,\\
      \frac{n+2}{l_1+nl_2} & \text{if }l_1\le 2l_2,\\
  \end{cases}
\end{equation}
if $D=D_w$, and
\begin{equation}\label{eq:a-invariant-DwDz}
     a(\Xc) = \max\left\{\frac{1}{l_2}, \frac{n+1}{l_1+nl_2}\right\} = \begin{cases}
      \frac{1}{l_2} & \text{if }l_1 > l_2,\\
      \frac{n+1}{l_1+nl_2} & \text{if }l_1\le l_2,\\
  \end{cases}
\end{equation}
if $D=D_w+D_z$. The adjoint class
\begin{equation}\label{eq:def-adjoint}
  E_\Xc = K_{X_n}+D+a(\Xc)L
\end{equation}
attaining this minimum is
\begin{equation}\label{eq:adjoint-1}
  E_\Xc = \begin{cases}
    (\frac{l_1}{l_2} - 2, 0) & \text{if } l_1>2l_2,\\
    (0, \frac{2l_2-l_1}{l_1+nl_2}) & \text{if }  l_1\le 2l_2
  \end{cases}
\end{equation}
if $D=D_w$, and 
\begin{equation}\label{eq:adjoint-2}
  E_\Xc = \begin{cases}
    (\frac{l_1}{l_2} - 1, 0) & \text{if } l_1>l_2, \\
    (0, \frac{l_2-l_1}{l_1+nl_2}) & \text{if } l_1\le l_2,
  \end{cases}
\end{equation}
if $D=D_w+D_z$.
In the latter cases, the adjoint class is trivial or a positive multiple of the class $(0,1)$ of the exceptional divisor of $\pi_n$, hence rigid. In the former cases, it is a positive multiple of the base point free class $(1,0)$, which is the degree of the monomials $a$, $c$, and $z$ in the Cox ring. This class induces the Iitaka fibration
\begin{equation}\label{eq:iitaka}
\phi\colon X_n \to \Pd^2,\ (a:b:c:d:z:w)\mapsto (a:c:z).
\end{equation}

Finally, the Fujita invariant $b(\Xc)$ is the codimension of the minimal face of the effective cone containing $E_\Xc$. It is thus $2$ if and only if both entries in the respective maximum in~(\ref{eq:a-invariant-Dw},\,\ref{eq:a-invariant-DwDz}) are equal and the corresponding adjoint bundle is trivial, and $1$ in any other case:
\begin{equation}\label{eq:b-invariant}
    b(\Xc) = \begin{cases}
  2 & \text{if $D=D_w$ and $l_1=2l_2$,}\\
  2 & \text{if $D=D_w+D_z$ and $l_1 = l_2$, and}\\
  1 & \text{otherwise.}
  \end{cases}
\end{equation}
\subsection{Height functions}\label{sec:height-functions}
For a polarization $L = (l_1+nl_2, l_2)$, the map
\begin{equation}\label{eq:height_torsor}
  \begin{aligned}
  H_L(a,b,c,d,z,w) = \max\left\{\begin{aligned}
      &\max\{|a|,|c|,|z|\}^{l_1}\max\{|b|,|d|\}^{l_2},\\ &\max\{|a|,|c|,|z|\}^{l_1+nl_2}|w|^{l_2}
  \end{aligned}\right\}
  \end{aligned}
\end{equation}
induces a height function on $X_n$ associated with $L$ by setting
\[
  H_L(x) = H_L(a,b,c,d,z,w)
\] for a point $x = (a:b:c:d:z:w)\in X_n(\Qd)$ represented in integral Cox coordinates $(a,b,c,d,z,w)\in \Ym_n(\Zd)$, in particular satisfying the coprimality condition~\eqref{eq:gcd}. Indeed, if $l_1$ and $l_2$ are both integers, then the set
\begin{equation*}
  \{\{a,c,z\}^{l_1}\{b,d\}^{l_2},\{a,c,z\}^{l_1+nl_2}w^{l_2}\}
\end{equation*}
of monomials of degree $L$ does not admit a common zero, and so these sections of a bundle of degree $L$ define a metric on $L$ and a height function; if they are not, then passing to an integral multiple $kL$ leads to the height function $H_L = H_{kL}^{1/k}$, using the linearity of the exponents in~\eqref{eq:height_torsor} in $L$. 

\section{The leading constants}\label{sec:leading-constants}
Throughout this section, let $\Xc = (X_n, D, L)$ be a polarized log variety consisting of one of the spherical varieties $X_n$ with $n\ge 2$, one of the $G$-invariant divisors $D$ obtained in~\eqref{eq:G-invariant-D}, and an ample line bundle $L=L_{(l_1, l_2)}$ metrized as in Section~\ref{sec:height-functions}.
For now, assume that $l_1 < 2l_2$ if $D=D_w$ and $l_1\le l_2$ if $D=D_w+D_z$; in particular, this means that the adjoint bundle $E_\Xc$ is rigid. We return to the moving case in Section~\ref{sec:moving-constants}.

\subsection{Archimedean Tamagawa measures}\label{sec:arch-measures}
Aiming to interpret the asymptotic formulas obtained in Section~\ref{sec:counting} geometrically, we start by describing the relevant types of Tamagawa measures. To this end, we shall make use of the chart
\begin{equation*}
  \psi \colon X_n \dashrightarrow \Ad^3,\quad (a:b:c:d:z:w)\mapsto (c/a, z/a, wa^n/b)
\end{equation*}
and its inverse
\begin{equation}\label{eq:local-coords-inverse}
  \tilde\psi\colon \Ad^3 \to X_n, \quad (c_0,z_0, w_0)\mapsto (1:1:c_0:c_0+z_0^{n+1}w_0:z_0:w_0).
\end{equation}
The descriptions~(\ref{eq:adjoint-1},\,\ref{eq:adjoint-2}) of the adjoint bundle (currently assumed to be rigid) motivate the definition
\begin{equation*}
  e_\Xc = \begin{cases}
  \frac{2l_2-l_1}{l_1+nl_2} & \text{if $D=D_w$ and} \\
  \frac{l_2-l_1}{l_1+nl_2} & \text{if $D=D_w+D_z$,} 
  \end{cases}
\end{equation*}
so that $E_\Xc = e_\Xc D_w$ (which is possibly zero); let $k>0$ be an integer such that $ke_\Xc \in \Zd$.
Appealing to~\eqref{eq:def-adjoint}, the meromorphic section
\begin{equation}\label{eq:abstract-L-section}
  1_{D_w}^{ke_\Xc}1_{D}^{-k}(\df c_0 \wedge \df z_0 \wedge \df w_0)^{-k}
\end{equation}
has degree $ka(\Xc)L$ and corresponds to the monomial
\begin{equation}\label{eq:concrete-L-section}
  w^{ke_\Xc} s_D^{-k} (a^{n-2}b^{-2})^{-k},
\end{equation}
where $s_D \in \{w, wz\}$ depending on $D$,
under suitable isomorphisms between the bundles $\Oc(D_w)$, $\Oc(D_z)$, and $\omega_X$, and the bundles whose meromorphic sections are homogeneous elements of $\Frac(R(X_n))$ of the respective degrees. (Note that up to scalars, $\df c_0 \wedge \df z_0 \wedge \df w_0$ is the unique section of the anticanonical bundle and $a^{n-2}b^{-2}$ the unique fraction of canonical degree $(-n,-2)$ with neither zeroes nor poles on $\im \tilde\psi$.)

As $\Xc$ is adjoint rigid, the adjoint divisor $E_\Xc$ is well-defined as a divisor (not just as a class). The relevant (residue) measures are supported on subvarieties corresponding to maximal faces of the analytic Clemens complex $\Cc^{\an}(D\setminus |E_\Xc|)$ of the complement of its support~\cite[after Def.~6.3]{Santens23}. As $\Cc^{\an}(D)$ is a simplex, so is this subcomplex; denote by $A\in \Cc^{\an}(D\setminus |E_\Xc|)$ its unique maximal face.

If $D=D_w$, we restrict ourselves to the case $l_1<2l_2$, in which this analytic Clemens complex is trivial with $A=\emptyset$ as its only face. Under this condition, the relevant Tamagawa measure is
\begin{equation*}
  \df \tau_{\Xc,\emptyset,\infty} =  \frac{\df c_0\, \df z_0\, \df w_0}{\norm{1_{D_w}}^{1-e_\Xc}\norm{\df c_0 \wedge \df z_0 \wedge \df w_0}},
\end{equation*}
the denominator being the $k$-th root of the norm of~\eqref{eq:abstract-L-section}. As $e_\Xc>0$ by the above assumption on $L$, this measure is finite. 
Using the correspondence with the monomial~\eqref{eq:concrete-L-section}, this measure can be expressed as
\begin{equation*}
  \frac{
    |a^{n-2}b^{-2}|^{-1}  \df c_0\, \df z_0\, \df w_0
    }{
      |w|^{1- e_\Xc} H_L(a,b,c,d,z,w)^{a(\Xc)}
    }
\end{equation*}
in Cox coordinates and as
\begin{equation*}
  \df \tau_{\Xc,\emptyset,\infty} =\frac{
    \df c_0\, \df z_0\, \df w_0
    }{
      |w_0|^{1- e_\Xc} H_L(1,1,c_0,c_0+z_0^{n+1}w_0,z_0,w_0)^{a(\Xc)}
    }
\end{equation*}
in local coordinates~\eqref{eq:local-coords-inverse}. Finally, let
\begin{equation}\label{eq:omega_infty_Dw}
  \omega_{\Xc,\infty} = \int_{X_n(\Rd)} \df \tau_{\Xc,\emptyset,\infty}
  =\int_{\Rd^3} \frac{
    \df c_0\, \df z_0\, \df w_0
    }{
      |w_0|^{1- e_\Xc} H_L(1,1,c_0,c_0+z_0^{n+1}w_0,z_0,w_0)^{a(\Xc)}
    }
\end{equation}
be the resulting volume of $X_n(\Rd)$.

Turning to $D=D_w+D_z$, the relevant analytic Clemens complex is a single vertex with maximal face $A=\{D_z\}$ if $l_1<l_2$ and a $1$-simplex with maximal face $A=\{D_z,D_w\}$ if $l_1 = l_2$.
In the former case, the Tamagawa measure
\begin{equation*}
  \df \tau_{\Xc,\{D_z\},\infty} = 2
  \frac{
    \df c_0\, \df w_0
  }{
    \norm{1_{D_w}}^{1-e_\Xc}\norm{\df c_0 \wedge \df w_0}
  }
\end{equation*}
is supported on $D_z(\Rd)$ and finite as $e_\Xc>0$. It includes a renormalization factor $2^{|A|}$. The norm of the $2$-form in the denominator is induced by one on $\omega_X(D_z)$ via the adjunction isomorphism. This isomorphism identifies $\df c_0\wedge \df w_0$ with the restriction of the $3$-form $z_0^{-1}\df c_0 \wedge \df z_0 \wedge \df w_0$. In particular, the denominator is
\begin{equation*} 
  \norm{1_{D_w}}^{1-e_\Xc} \norm{z_0^{-1}\df c_0 \wedge \df z_0 \wedge \df w_0} = \lim_{z_0\to 0} \frac{\norm{1_{D_w}}^{1-e_\Xc}\norm{1_{D_z}\df c_0 \wedge \df z_0 \wedge \df w_0}}{|z_0|}.
\end{equation*} 
Once more using the correspondence between~\eqref{eq:abstract-L-section} and~\eqref{eq:concrete-L-section}, changing into local coordinates, and passing to the limit, the measure can thus be explicitly described as
\begin{equation*}
  \df \tau_{\Xc,\{D_z\},\infty} = 2\frac{
    \df c_0\, \df w_0
    }{
      |w_0|^{1- e_\Xc} H_L(1,1,c_0,c_0,0,w_0)^{a(\Xc)}
    }
\end{equation*}
in local coordinates; denote by
\begin{equation}\label{eq:omega_infty_DwDz_1}
  \omega_{\Xc,\infty} = \int_{D_z(\Rd)} \df \tau_{\Xc,\{D_z\},\infty}
  = 2\int_{\Rd^2} \frac{
    \df c_0\, \df w_0
    }{
      |w_0|^{1- e_\Xc} H_L(1,1,c_0,c_0,0,w_0)^{a(\Xc)}
    }
\end{equation}
the total volume.

If $l_1=l_2$, that is, if $E_\Xc=0$, this measure is infinite. However, as the analytic Clemens complex is larger in this case, the relevant measure is a residue measure on the real points of the boundary stratum $Z_A = D_z\cap D_w$ corresponding to the maximal face $A=\{D_w,D_z\}$ of the Clemens complex. Analogously to the previous cases, it can be determined to be
\begin{equation*}
  \df \tau_{\Xc,\{D_w,D_z\},\infty} = 4 \frac{
    \df c_0
    }{
      H_L(1,1,c_0,c_0,0,0)^{a(\Xc)}
    },
\end{equation*}
again including the renormalization factor $2^{|A|}$; once more, let
\begin{equation}\label{eq:omega_infty_DwDz_2}
  \omega_{\Xc,\infty} = \int_{D_w\cap D_z (\Rd)} \df \tau_{\Xc,\{D_w,D_z\},\infty}= 4 \int_{\Rd}\frac{
    \df c_0
    }{
      H_L(1,1,c_0,c_0,0,0)^{a(\Xc)}
    }
\end{equation}
be the total volume.

\subsection{Non-archimedean Tamagawa measures}

For finite places, the measures of interest
\begin{equation*}
  \df\tau_{\Xc,p} = \frac{\norm{1_{D_w}}^{e_\Xc}_p \df c_0\, \df z_0\, \df w_0}{\norm{1_{D}}_p\norm{\df c_0 \wedge \df z_0 \wedge \df w_0}_p}
\end{equation*}
are supported on $\Um(\Zd_p)$,
and we denote by
\begin{equation*}
  \omega_{\Xc,p} = \lambda_{\Xc,p} \int_{\Um(\Zd_p)} \df\tau_{\Xc,p}
\end{equation*}
the volume of the set of $\Zd_p$-integral points with a convergence factor $\lambda_{\Xc,p}$ as in \cite[Def.~2.8]{CLT10}.
The sections $1_{D_w}$, $1_{D_z}$, and $\df c_0\wedge \df z_0\wedge \df w_0$ are integral and primitive; as the same holds true for $w$, $z$, and $a^{n-2}b^{-2}$, the isomorphisms appearing in the beginning of Section~\ref{sec:arch-measures} spread out to the models $\Xm_n$ and $\Xm_n\setminus V(ab)$, respectively. As a consequence of the coprimality conditions on the torsor, the $p$-adic analogue of the height~\eqref{eq:height_torsor} evaluates to $1$ on all $\Zd_p$-points.
In particular, the corresponding adelic metrics are induced by the model $\Xm_n$ at all finite places. Consequently, evaluating $\norm{1_{D}}_p$ in any $\Zd_p$-point on $\Um$ results in $1$; as $E_\Xc\subset D$ for all $n$, $D$, and $L$, the same holds true for $\norm{1_{D_w}}_p$.  

The fiber over any prime $p$ is cut out by the polynomial $ad-bc-z^{n+1}w$ in the toric $\Fd_p$-variety described by $\Sigma_n$. Its universal torsor is \'etale locally trivial, hence \'etale locally isomorphic to $(\Zm_n)_{\Fd_p} \times \Gd^2_{\mathrm{m},\Fd_p}$. The preimage of the subvariety $(\Xm_n)_{\Fd_p}$ is thus \'etale locally isomorphic to $(\Xm_n)_{\Fd_p} \times \Gd^2_{\mathrm{m},\Fd_p}$. This last variety is readily checked to be smooth outside $V(a,c,z)$ using the Jacobi criterion, and the same criterion then implies the smoothness of its factor $(\Xm_n)_{\Fd_p}$.
Therefore, the identity
\begin{equation*}
\omega_{\Xc,p} = \lambda_{\Xc,p} \frac{\#\Um(\Fd_p)}{p^3}
\end{equation*}
that a priori holds for almost all primes $p$ (see e.g.~\cite[\S\,2.4.1]{CLT10}) in fact holds for all of them.
\begin{lemma}
  The number of $\Fd_p$-points on the integral models is
\begin{equation*}
  \#\Xm_n(\Fd_p) = p^3 + 2p^2 +2p+ 1.
\end{equation*}
\end{lemma}
\begin{proof}
  This number can be obtained elementarily by counting points on the universal torsor and dividing by $\#\Gdm^2(\Fd_p) = (p-1)^2$ to obtain the number of fibers or from an application of the Lefschetz trace formula as in~\cite[Lem.~2.2.4]{Peyre95Duke}:
  Peyre's arguments yield
  \begin{align*}
    H_{\et}^0(X_{\Qbar},\Qd_l)&\cong H_{\et}^6(X_{\Qbar},\Qd_l) \cong \Qd_l,\\
    H_{\et}^1(X_{\Qbar},\Qd_l)&\cong H_{\et}^5(X_{\Qbar},\Qd_l) =0, \quad\text{and}\\
    H_{\et}^2(X_{\Qbar},\Qd_l)&\cong H_{\et}^4(X_{\Qbar},\Qd_l)=\Pic(X_{\Qbar})\otimes_\Zd \Qd_l \cong \Qd_l^2
  \end{align*}
  with a trivial action by the Galois group using Poincaré duality, leaving only degree $3$. 
  The exponential sequence results in an exact sequence
  \[
    \Br(X(\Cd)) = H^2(X(\Cd), \Oc^\times_{X(\Cd)}) \to H^3(X(\Cd),\Zd) \to H^3(X(\Cd),\Oc_{X(\Cd)});
  \]
  the group on the left is trivial as $X_\Cd$ is rational, and the one on the right is so by Kodaira vanishing.
  Hence, the cohomology group in the middle vanishes. Then so does $l$-adic cohomology
  $H_\et^3(X_{\Qbar}, \Qd_l)$ for any prime $l$.
  The smooth base change theorem now implies that these groups coincide with the $l$-adic cohomology groups $H_{\et}^i((\Xm_n)_{\overline{\Fd_p}}, \Qd_l)$ for $l\ne p$.
\end{proof}

For the following arguments to count points on the two relevant open subvarieties,
let $\Dm_w = \overline{D_w}$ be the closure in $\Xm$, and analogously for the remaining generators of the Cox ring.
The map
\begin{equation*}
  \Dm_w \to \Pd^1,\quad (a:b:c:d:z:0)\mapsto(b:d)
\end{equation*}
makes $\Dm_w$ an (in fact trivial) $\Pd^1$-bundle over $\Pd^1$. Hence, $\#\Dm_w(\Fd_p)=(p+1)^2$ and
\begin{equation*}
  \frac{\#\Um(\Fd_p)}{p^3}
  = \frac{\#\Xm_n(\Fd_p) - \# \Dm_w(\Fd_p) }{p^3}
  = 1+\frac{1}{p}
\end{equation*}
if $D=D_w$.

As integral points on $D_w$ and on $D_z$ satisfy the same torsor equation with symmetric coprimality conditions, their numbers are equal. The intersection $D_w\cap D_z$ is a line on $D_w$ with $p+1$ points, whence
\begin{equation}\label{eq:points-away-from-Dw-and-Dz}
  \frac{\#\Um(\Fd_p)}{p^3}
  = \frac{\#\Xm_n(\Fd_p) - \# \Dm_w(\Fd_p)-\# \Dm_z(\Fd_p) + \# (\Dm_w\cap\Dm_z)(\Fd_p) }{p^3}
  = 1 - \frac{1}{p^2}
\end{equation}
if $D=D_w+D_z$.

Finally, as the Galois group acts trivially, the convergence factors are simply
\begin{equation*}
    \lambda_{\Xc,p} = \left(1 - \frac 1 p\right)^{\rk \Pic U},
\end{equation*}
resulting in
\begin{equation}\label{eq:omega_p}
    \omega_{\Xc,p} = 1 - \frac{1}{p^2}
\end{equation}
for both choices of $D$.

\subsection{Adelic Tamagawa measures}

Again denoting by $A$ the unique maximal face of the analytic Clemens complex, multiplying these local measures results in the adelic Tamagawa measure 
\begin{equation*}
    \tau_\Xc = \tau_{\Xc,A,\infty} \prod_{p} \lambda_{\Xc,p}\tau_{\Xc, p}
\end{equation*}
on $X(\Ad_\Qd)$. It is supported on $Z_A(\Rd) \times \prod_p \Um(\Zd_p)$, where $Z_A = \bigcap_{E\in A}E$ is the boundary stratum corresponding to the face $A$.
The total volume of the space of adelic points is
\begin{equation*}
\omega_\Xc = \tau_\Xc(X(\Ad_\Qd)) = \omega_{\Xc,\infty}\prod_p\omega_{\Xc, p}.
\end{equation*}

\subsection{Effective cones}\label{sec:alpha}

The remaining ingredient for asymptotic formulas is a constant $\alpha(\Xc)$ related to the geometry of effective cones. As the base field $\Qd$ only has one archimedean place, the relevant cone resides in the Picard group $\Pic(U_A)$ of the subvariety $U_A = X\setminus \bigcap_{E\subseteq A}E$ associated with the unique maximal face $A$ of the analytic Clemens complex $\Cc^\an(D\setminus |E_\Xc|)$.

If $L$ is a multiple of the log anticanonical bundle, this Picard group is $\Pic(X_n)\cong \Zd^2$ with basis $([D_w],[D_z])$; its effective cone is smooth with the same basis (as it is generated by the classes of the generators of the Cox ring), and thus the 
representation $L=l_2[D_w]+(l_1+nl_2)[D_z]$ results in
\begin{equation}\label{eq:alpha_loganticanonical}
  \alpha(\Xc) = \frac{1}{(l_1+nl_2)l_2};
\end{equation}
see e.g.~\cite[Rmk.~2.2.9]{wilsch-toric}.

Otherwise, $[D_z]$ is a basis of both $\Pic(U_a) = \Pic(X_n\setminus D_w)$ and its smooth effective cone, and 
$L|_{X_n\setminus D_w} = (l_1+nl_2)[D_z]$ results in
\begin{equation}\label{eq:alpha_not_loganticanonical}
  \alpha(\Xc) = \frac{1}{l_1+nl_2}.
\end{equation}

\begin{remark}
  The Picard group $\Pic(U_A)$ corresponds to $\Pic(U;A)$ in the notation of~\cite[Def.~2.2.1, Rmk.~2.2.9\,(i)]{wilsch-toric} and $\Pic((X;A)_L)$ in the notation of~\cite[\S\,6.2]{Santens23}.
  The constant $\alpha(\Xc)$ coincides with $\theta((X; A),L)/(b(\Xc)-1)!$ as in~\cite[Def.~6.10]{Santens23} (but note that $(b(\Xc)-1)!=1$ in any case).
\end{remark}

\subsection{Moving adjoint divisors}\label{sec:moving-constants}
Finally, let $l_1>2l_2$ if $D=D_w$ and $l_1>l_2$ if $D=D_w+D_z$; that is, assume that the adjoint divisor $E_\Xc$ is not rigid.
For each point $t = (a:c:z)\in \Pd^2(\Zd)$ represented by coprime integers, the fiber $\Um_t = \phi^{-1}(t)\subseteq \Um$ is either empty
(for $D=D_w$, this happens if $\gcd(a,c)\ne 1$; for $D=D_w+D_z$, this happens if $\gcd(a,c)\ne 1$ or $z\ne \pm 1$)
or isomorphic to $\Ad^1_\Zd$. On each nonempty fiber, constructions analogous to the ones above, using the metrized bundle $L|_{\Um_t}$ in place of $L$, result in Tamagawa volumes $\omega_{\Xc,v}(t)$ (supported on $X_t \cap D_w$ for the archimedean place $v=\infty$) and effective cone constants $\alpha(\Xc_t)$. If a fiber is empty, define these numbers to be $0$.

\section{Counting}\label{sec:counting}

Keep the notation of the previous section; in particular, let $\Xc=(\Xm_n, D, L)$ be a polarized log variety consisting of the model $\Xm_n$ of the spherical variety $X_n$ for an integer $n\ge 2$, one of the $G$-invariant divisors $D\in\{D_w, D_w + D_z\}$ that make $(X_n,D)$ log Fano, and an ample line bundle $L=(l_1+nl_2, l_2)$. These data are treated as fixed, and all implied constants are allowed to depend on them. We are interested in the asymptotic behavior of the counting function
\begin{equation*}
  N_{\Xc,\emptyset}(B) = \#\{x \in \Um(\Zd) \mid H_L(x)\le B\}
\end{equation*}
for $B\ge 2$.
By the description of rational points in Section~\ref{sec:cox}, this function can be expressed more explicitly as
\begin{equation}\label{eq:counting-concrete}
  N_{\Xc,\emptyset}(B) =
  \begin{cases}
    \frac{1}{2} \#\left\{(a, b, c, d, z) \in \Zd^5\ \middle|\ \substack{ad-bc=z^{n+1},\\ \gcd(a,c) = 1,\\ H_L(a,b,c,d,z,1) \le B}\right\} & \text{if } D=D_w, \text{ and}\\[1em]
    \#\left\{(a, b, c, d) \in \Zd^4\ \middle|\  \substack{ad-bc=1,\\ \gcd (a, c)=1,\\ H_L(a, b, c, d,1,1) \le B}\right\}  & \text{if } D=D_w+D_z.
  \end{cases}
\end{equation}
Indeed, integral points are precisely those that satisfy $w=\pm1$ and $w=\pm1$, $z=\pm 1$, respectively. Setting these coordinates to $1$ turns the $4$-to-$1$-correspondence for rational points into a $2$-to-$1$-correspondence and $1$-to-$1$-correspondence, respectively. In particular, the restraint $w=1$ makes the second coprimality condition in~\eqref{eq:gcd} automatic, while the torsor equations $ad-bc=1$ and $ad-bc=z^{n+1}$ (together with $\gcd(a,c,z)=1$ in the latter case) imply 
$\gcd(a,c)=1$, which is stronger than the coprimality condition in~\eqref{eq:gcd}.

\subsection{The first summation}

We begin by estimating the number of solutions involving a fixed vector $(a,c,z)\in \Zd^3$ by means of a one-dimensional volume. Concretely, for a point $t=(a:c:z)\in \Pd^2(\Qd)$ represented by a primitive vector $(a,c,z)\in \Zd^3_\prim$, let $N_{\Xc_t,\emptyset}(B)$ be the number of integral points in the fiber $\phi^{-1}(t)$ of~\eqref{eq:iitaka} of $L$-height at most $B$, that is, the number of points in the respective set in~\eqref{eq:counting-concrete} with fixed $a$, $c$, and $z$. To define a volume function $ V_\Xc(a, c, z; B)$ for $(a,c)\in \Rd^2\setminus \{(0,0)\}$, $z\in \Rd$, and $B\ge 2$, let
\begin{equation*}
    V_\Xc(a, c, z; B) = \int_{|b|, |a^{-1}(z^{n+1}+bc)| \le B^{1/l_2}\max\{|a|, |c|, |z|\}^{-l_1/l_2}}\frac{\df b}{|a|}
\end{equation*}
if $a\ne 0$, and $V_\Xc(0, c, z; B) = V_\Xc(c, 0, z; B)$. Note that for $(a,c)\in \Rd^2\setminus(0,0)$, the linear change of variables $d=a^{-1}(z^{n+1}+bc)$  --- or the definition if $a=0$ --- yields the identity $V_\Xc(a, c, z; B) = V_\Xc(c, a, z; B)$.

\begin{lemma}\label{lem:summation-b}
    For $t = (a:c:z)\in \Pd^2(\Qd)$ represented by coprime integers $a$, $c$, and $z$, the number of integral points of bounded height in the fiber $\phi^{-1}(t)$ is
    \begin{equation*}
      N_{\Xc_t,\emptyset} (B)= V_\Xc(a,c,z;B) + O(1)
    \end{equation*}
    whenever it is nonzero, that is, whenever $D=D_w$ and $\gcd(a,c)=1$ or $D=D_w+D_z$, $\gcd(a,c)=1$, and $z\in \Zd^\times$.

    Moreover, the total number of integral points of bounded height is
    \begin{equation}\label{eq:after-first-summation-1}
        N_{\Xc,\emptyset}(B) = \frac{1}{2}\sums{a,c,z \in \Zd\\|a|,|c|,|z| \le B^{1/(l_1+nl_2)}\\\gcd(a,c)=1} V_\Xc(a,c,z;B) + O(B^{3/(l_1+nl_2)})
    \end{equation}
    if $D=D_w$, and
    \begin{equation}\label{eq:after-first-summation-2}
        N_{\Xc,\emptyset}(B) = \sums{a,c \in \Zd\\|a|,|c| \le B^{1/(l_1+nl_2)}\\\gcd(a,c)=1} V_\Xc(a,c,1;B) + O(B^{2/(l_1+nl_2)}),
    \end{equation}
    if $D=D_w+D_z$.
\end{lemma}
\begin{proof}
   For fixed $(a,c,z)\in \Zd^3$ with $\gcd(a,c)=1$ and $a\ne0$, the coordinate
   \[
      d=(bc+z^{n+1})/a,
   \]
   is integral if and only if $b\equiv -c^{-1}z^{n+1} \pmod a$. As a consequence, the number of $(b,d)\in\Zd^2$ satisfying $H_L(a,b,c,d,z,1)\le B$ and the torsor equation is 
   \begin{equation*}
       V_\Xc(a,c,z;B)+O(1),
   \end{equation*}
   with an analogous argument treating the case $c\ne0$, settling the first part of the lemma.
   Summing this error term over $a$ and $c$ results in
   \[
    \sums{a,c \in \Zd\\|a|,|c| \le B^{1/(l_1+nl_2)}\\\gcd(a,c)=1} 1 \ll B^{2/(l_1+nl_2)}
   \]
   and the second part of the lemma after specializing to $z=1$; further summing over $z$ results in the first part. (The factor $1/2$ stemming from above is explained by the fact that each fiber is presented by two vectors $(a,c,z)$ in the first case and by a unique vector $(a,c,1)$ in the second case.)
\end{proof}

\begin{lemma}\label{lem:compare_V_V'}
  Let $(a,c,z) \in \Rd^3$ with $(a,c) \ne (0,0)$. Let
  \begin{equation}\label{eq:V'}
    V'_\Xc(a,c,z) = \int_{|b|,|a^{-1}bc| \le \max\{|a|,|c|,|z|\}^{-l_1/l_2}}\frac{\df b}{|a|}
  \end{equation}
  if $a \ne 0$, and $V'_\Xc(0,c,z) = V'_\Xc(c,0,z)$.
  Then
  \begin{equation}\label{eq:compare-V-V'}
    V_\Xc(a,c,z;B) = V'_\Xc(a,c,z) B^{1/l_2} + 
    \begin{cases}
        O(|z|^{n+1}/|ac|) &\text{if $ac \ne 0$,}\\
        O(|z|^{n+1}/(|a|^2-|ac|)) &\text{if $|c|<|a|$,}\\
        O(|z|^{n+1}/(|c|^2-|ac|)) &\text{if $|a|<|c|$,}\\
    \end{cases}
  \end{equation}
  where each applicable error term suffices on its own whenever more than one of the conditions is met.
\end{lemma}

\begin{proof}
  By symmetry, we may assume $a \ne 0$. Setting
  \begin{equation*}
    Z = B^{1/l_2}\max\{|a|,|c|,|z|\}^{-l_1/l_2},
  \end{equation*}
  the first step is to bound the difference
  \begin{equation}\label{eq:diff-V-V'}
    \left|  V_\Xc(a,c,z;B) - \int_{|b|,|a^{-1}bc| \le Z} \frac{\df b}{|a|} \right|.
  \end{equation}
  
  If $c \ne 0$ and $a/c>0$, this means comparing an integral over those $b\in \Rd$ satisfying
  $-Z \le b \le Z$ and
  \begin{equation*}
    -Zac^{-1}-z^{n+1}c^{-1} \le b \le Zac^{-1}-z^{n+1}c^{-1}
  \end{equation*}
  with one over the set of $b$ satisfying $-Z \le b \le Z$ and
  $-Zac^{-1} \le b \le Zac^{-1}$; if $a/c<0$, the comparison is similar. In each of these cases,
  the difference of the two sets is contained in the set of $b$ with
  \begin{equation*}
    Z|ac^{-1}|-|z^{n+1}c^{-1}| \le |b| \le Z|ac^{-1}|+|z^{n+1}c^{-1}|,
  \end{equation*}
  which form two intervals of length $O(|z^{n+1}c^{-1}|)$; multiplying with the integrand $|a|^{-1}$ results in the first error term in~\eqref{eq:compare-V-V'}. 

  To obtain the second bound for the error term, assume $|c|<|a|$. If
  \begin{equation*}
    Z|ac^{-1}|-|z^{n+1}c^{-1}| \ge Z,
  \end{equation*}
  both integration domains in~\eqref{eq:diff-V-V'} are given by $|b| \le Z$, hence the difference is $0$ in this case. Otherwise, $Z < |z^{n+1}|/(|a|-|c|)$, and we end up with an error term
  \begin{equation*}
      \int_{|b|\le Z} \frac{\df b}{|a|} \ll \frac{|z^{n+1}|}{|a^2|-|ac|}
  \end{equation*}
  as stated. The final case $|a|<|c|$ is symmetric.
  
  In each case, the main term arises from the integral on the right of~\eqref{eq:diff-V-V'} by the change of variables $b\mapsto B^{-1/l_2}b$.
\end{proof}

\begin{lemma}\label{lem:sum_V'}
  If $D=D_w$, then
  \begin{equation}\label{eq:sum-1-acz-unextended}
    N_{\Xc,\emptyset}(B) = \frac 1 2 \sums{(a,c,z)\in \Zd^2_\prim \times \Zd \\ |a|,|c|,|z|\le B^{1/(l_1+nl_2)}} V'_\Xc(a,c,z)B^{1/l_2} + O(B^{(n+2)/(l_1+nl_2)}\log B).
  \end{equation}
  If $D=D_w+D_z$, then
  \begin{equation*}
    N_{\Xc,\emptyset}(B) = \sums{(a,c)\in \Zd^2_{\prim} \\ |a|,|c|\le B^{1/(l_1+nl_2)}}
    V'_\Xc(a,c,1)B^{1/l_2} + O(B^{2/(l_1+nl_2)}).
  \end{equation*}
\end{lemma}

\begin{proof}
  To arrive at~\eqref{eq:sum-1-acz-unextended}, we may assume $|c| \le |a|$ by the symmetry of all expressions involved. Summing the error term~\eqref{eq:diff-V-V'} over $a$, $c$, and $z$ results in
  \begin{align*}
  &\sums{|a|,|z|\le B^{1/(l_1+nl_2)}\\|a|\ge 1} \left(\sum_{|c|\le |a/2|}\frac{|z|^{n+1}}{|a|^2-|ac|}+\sum_{|a/2| < |c| \le |a|} \frac{|z|^{n+1}}{|ac|} \right)\\
  &\quad \ll \sums{|a|,|z|\le B^{1/(l_1+nl_2)}\\|a|\ge 1} \frac{|z^{n+1}|}{|a|}
  \ll B^{(n+2)/(l_1+nl_2)} \log B.
  \end{align*}
  Setting $z=1$ and only summing over $a$ and $c$ to treat the second case results in the error term $O(\log B)$ that is dominated by the error term in Lemma~\ref{lem:summation-b}.
\end{proof}

\subsection{Moving adjoint divisors: extending the sum}
For now, we investigate those polarizations $L$ for which the adjoint divisor $E_\Xc$ computed in equations~(\ref{eq:adjoint-1},~\ref{eq:adjoint-2}) is a positive multiple of $(1,0)$, that is, moving.

For $t = (a:c:z)\in \Pd^2(\Qd)$ with coprime integral coordinates, let
\begin{equation*}
    c'(\Xc_t) = \begin{cases}
        V_\Xc'(a, c, z) & \text{if $D=D_w$ and $\gcd(a,c)=1$,} \\
        V_\Xc'(a, c, 1) & \text{if $D=D_w+D_z$ and $\gcd(a,c)=z=1$,} \\
        0 & \text{otherwise.}
    \end{cases}
\end{equation*}
These constants can further be evaluated to
\begin{equation}\label{eq:constants-moving-explicit}
  \frac{1}{\max\{|a|,|c|,|z|\}^{l_1/l_2}\max\{|a|,|c|\}}
\end{equation}
whenever they are nonzero by integrating~\eqref{eq:V'}.
In light of Lemmas~\ref{lem:summation-b} and~\ref{lem:compare_V_V'}, the number of points on each fiber $t\in \Pd^2(\Qd)$ is thus
\begin{equation}\label{eq:points_on_fiber}
  N_{\Xc_t,\emptyset} = c'(\Xc_t)B^{1/l_2} + O_t(1).
\end{equation}

\begin{prop}\label{prop:count-moving}
  If the adjoint divisor is a positive multiple of $(1,0)$, then
    \begin{equation}\label{eq:sum-extended}
        N_{\Xc,\emptyset}(B) = \sum_{t\in \Pd^2(\Qd)} c'(\Xc_t) B^{1/l_2} + O(B^{1/l_2-\delta}(\log B)^2)
    \end{equation}
    with
    \begin{equation*}
      \delta = \begin{cases}
        (l_1/l_2 - 2)/(l_1+nl_2) & \text{if $D=D_w$ and} \\
        (l_1/l_2 - 1)/(l_1+nl_2) & \text{if $D=D_w+D_z$.}
      \end{cases}
    \end{equation*}
\end{prop}

\begin{proof}
  For $D=D_w$, the error arising from removing the height condition
  \begin{equation*}
    |a|,|c|,|z|\le B^{1/(l_1+nl_2)}
  \end{equation*}
  in~\eqref{eq:sum-1-acz-unextended} is at most
  \begin{align*}
    \sums{(a,c)\in \Zd^2_\prim,\,z\in \Zd \\ \max\{|a|,|c|,|z|\}> B^{1/(l_1+nl_2)}} V'_\Xc(a,c,z)B^{1/l_2}
    &\ll \sums{a,c,z\ge 0 \\ a\le c,\ c\ne 0 \\ \max\{c,z\}> B^{1/(l_1+nl_2)}} \frac{B^{1/l_2}}{\max\{c,z\}^{l_1/l_2}c} \\
    &\ll \sums{c,z\ge 0,\\ \max\{c,z\}> B^{1/(l_1+nl_2)}} \frac{B^{1/l_2}}{\max\{c,z\}^{l_1/l_2}},
  \end{align*}
  using~\eqref{eq:constants-moving-explicit} and its symmetry in $a$ and $c$.
  Using the symmetry of this expression in $c$ and $z$, this is at most a constant multiple of
  \begin{equation*}
    \begin{aligned}
    \sums{0\le c\le z,\\ z> B^{1/(l_1+nl_2)}}\frac{B^{1/l_2}}{z^{l_1/l_2}}
    &\ll \sum_{z> B^{1/(l_1+nl_2)}} \frac{B^{1/l_2}}{z^{l_1/l_2-1}}
    \ll B^{1/l_2} B^{(2-l_1/l_2)/(l_1+nl_2)}
    \end{aligned}
  \end{equation*}
  as $l_1/l_2>2$. In particular, the sum in equation~\eqref{eq:sum-extended} converges.
  
  Turning to $D=D_w+D_z$, an analogous argument bounds the error term by a constant multiple of 
  \begin{equation*}
    \sum_{c>B^{1/(l_1+nl_2)}}\frac{B^{1/l_2}}{c^{l_1/l_2}} \ll B^{1/l_2} B^{(1-l_1/l_2)/(l_1+nl_2)} 
  \end{equation*}
  as $l_1/l_2>1$.

  Finally, note that for $D=D_w$, each point $(a:c:z)\in \Pd^2(\Qd)$ corresponds to two vectors $(a,c,z)\in \Zd^2_\prim\times \Zd$, resulting in an additional factor of $2$, while for $D=D_w+D_z$, each of the relevant points $(a:c:1)\in \Pd^2(\Qd)$ corresponds to precisely one vector $(a,c)\in \Zd^2_\prim$.
\end{proof}

\subsection{Rigid adjoint divisors: the remaining summations}

Now, let $l_1 < 2l_2$ for $D=D_w$ and $l_1 \le l_2$ for $D=D_w+D_z$; in particular, $E_\Xc$ is rigid. In the latter case, this includes $E_\Xc=0$, in which $L$ is a multiple of the log anticanonical bundle. In the former case, the log anticanonical bundle $L=(n+2,2)$ with $l_1=2$ and $l_2=1$ has already been treated by Chambert-Loir and Tschinkel~\cite{CLT12} (see Remark~\ref{rmk:equivariant}); statements for its multiples with arbitrary $l_1 =2l_2$ can easily be deduced from their theorem.

For both boundaries, the first step in the remaining estimates is to remove the coprimality condition in~\eqref{eq:after-first-summation-1} and \eqref{eq:after-first-summation-2} by means of a Möbius inversion: for fixed $z\ne 0$, the sum over coprime $a$ and $c$ in either equation is
\begin{equation*}
  \sums{|a|,|c|\le B^{1/(l_1+nl_2)}\\(a,c)\ne (0,0)} \sum_{k\mid (a,c)} \mu(k) V_\Xc(a,c,z;B),
\end{equation*}
which can be rewritten as
\begin{equation}\label{eq:moebius}
  \sum_{k \ge 1} \mu(k) \sums{|a|,|c|\le B^{1/(l_1+nl_2)}/k\\(a,c)\ne (0,0)}V_\Xc(ka,kc,z;B).
\end{equation}

\begin{prop}\label{prop:rigid_integrals}
    Let
    \begin{align*}
        W_\Xc(B) &= \begin{cases}
          \displaystyle{\frac 1 2 \ints{H_L(a,b,c,(bc+z^{n+1})/a,z,1) \le B} \frac{\df a\,\df b\,\df c\,\df z}{|a|}} & \text{if $D=D_w$,}\\[1em]
          \displaystyle{\int_{H_L(a,b,c,bc/a,0,1) \le B} \frac{\df a\,\df b\,\df c}{|a|}} & \text{if $D=D_w+D_z$, $l_1<l_2$,} \\[1em]
          \displaystyle{\ints{H_L(a,b,c,bc/a,0,1) \le B\\\max\{|a|,|c|\} \ge 1} \frac{\df a\,\df b\,\df c}{|a|}} & \text{if $D=D_w+D_z$, $l_1=l_2$.}
        \end{cases}
    \end{align*}
    Then
    \begin{align*}
        N_{\Xc,\emptyset}(B) &= \frac{1}{\zeta(2)}W_\Xc(B) + 
        \begin{cases}
            O\!\left(B^{a(\Xc)-\delta(\Xc)}(\log B)^3\right), &\text{if $D=D_w$,}\\
            O\!\left(B^{1/l_2}\right), &\text{if $D=D_w+D_z$,}
        \end{cases}
    \end{align*}
    with the saving
    \begin{equation}\label{eq:delta-nL}
      \delta(\Xc) = \frac{\min\{l_1,2l_2-l_1\}}{l_2(l_1+nl_2)}
    \end{equation}
    in the former case.
\end{prop}

\begin{proof}
    The error terms in Lemma~\ref{lem:sum_V'} dominate the main term if $D=D_w$, and so our starting point is Lemma~\ref{lem:summation-b} combined with \eqref{eq:moebius}. Observe that~\eqref{eq:constants-moving-explicit} is still valid as an upper bound, though; that is,
    \begin{equation}\label{eq:V_nl-bound}
        V_\Xc(a,c,z;B) \ll 
            \frac{B^{1/l_2}}{\max\{|a|,|c|,|z|\}^{l_1/l_2}\max\{|a|,|c|\}}.
    \end{equation}

    \subsubsection*{Summation over $a$}
    By~\cite[Lem.~3.6]{df14}, $V_\Xc$ treated as a function in $a$ is piecewise monotonous, where the number of pieces is bounded independently of $c$, $z$, and $B$; by~\cite[Lem.~3.1]{der09}, the sum over $a$ can thus be estimated by an integral, introducing an error
  
    \begin{equation}\label{eq:replacing-sums-by-integrals}
      \begin{aligned}
        \left\lvert \sums{|ka|\le B^{1/(l_1+nl_2)}\\ \max\{|a|,|c|\}\ge 1}
        V_\Xc(ka,kc,z;B)
        -
        \ints{|ka|\le B^{1/(l_1+nl_2)}\\ \max\{|a|,|c|\}\ge 1}
        V_\Xc(ka,kc,z;B)\df a\right\rvert \\[.5em]
         \ll \sup_{\substack{|ka|\le B^{1/(l_1+nl_2)} \\ \max\{|a|,|c|\}\ge 1 }} V_\Xc(ka,kc,z;B).
      \end{aligned}
    \end{equation}
    By the linear change of variables $a\mapsto a/k$, the integral in~\eqref{eq:replacing-sums-by-integrals} can be rewritten as $k^{-1}V_\Xc^{(1)}(k,kc,z;B)$ with
    \begin{equation*}
      V_\Xc^{(1)}(k,c,z;B) = \ints{|a|\le B^{1/(l_1+nl_2)} 
      \\ \max\{|a/k|,|c/k|\}\ge 1} V_\Xc(a,c,z;B) \df a
    \end{equation*}
    (which vanishes for $k > B^{1/(l_1+nl_2)}$). Using \eqref{eq:V_nl-bound}, the second line of~\eqref{eq:replacing-sums-by-integrals} is at most
    \begin{equation*}
      \frac{B^{1/l_2}}{\max\{k,|kc|,|z|\}^{l_1/l_2} \max\{k,|kc|\}}
    \end{equation*}
    since the supremum ranges only over $|a|\ge 1$ for $c=0$, while the terms $k$ do not contribute to the maximum for $c \ne 0$. If $D=D_w$, its sum over $c$, $z$ and $k$ can be bounded by a constant multiple of
    \begin{equation}\label{eq:error-N1}
      B^{\frac{n+2}{l_1+nl_2} - \frac{\min\{l_1,2l_2-l_1\}}{l_2(l_1+nl_2)}}(\log B)^3,
    \end{equation}
    by cutting it into three regions to simplify the maximum in the denominator; this bound can be rewritten as $B^{a(\Xc)-\delta(\Xc)}(\log B)^3$ by~\eqref{eq:a-invariant-Dw} and~\eqref{eq:delta-nL}.

    If $D=D_w+D_z$ (hence $z=1$), the error term is bounded by
    \begin{equation*}
      \sum_{k, c\in \Zd,\ k\ge 1} \frac{B^{1/l_2}}{k^{l_1/l_2+1} \max\{1,|c|\}^{l_1/l_2+1}} \ll B^{1/l_2},
    \end{equation*}
    whence
    \begin{equation*}
      N_{\Xc,\emptyset}(B) = \sums{k, c \in\Zd,\ k\ge 1 \\ |kc| \le B^{1/(l_1+nl_2)}} \frac{\mu(k)}{k} V_\Xc^{(1)}(k,kc,1;B) + O(B^{1/l_2}).
    \end{equation*}
    For $D=D_w$, an analogous expression involving an additional sum over $|z| \le B^{1/(l_1+nl_2)}$, a factor $1/2$, and the error term~\eqref{eq:error-N1} holds.

    \subsubsection*{Summation over $c$}
    Arguing analogously to before and setting
    \begin{equation*}
      V_\Xc^{(2)}(k,z;B) = \ints{|c|\le B^{1/(l_1+nl_2)}} V_\Xc^{(1)}(k,c,z;B) \df c,
    \end{equation*}
    the difference 
    \begin{equation*}
      \left| \frac{1}{k}\sum_{|c|\le B^{1/(l_1+nl_2)}/k}
      V_\Xc^{(1)}(k,kc,z;B) - \frac{1}{k^2}V_\Xc^{(2)}(k,z;B)
      \right|
    \end{equation*}
    is at most $\sup_{|c|\le B^{1/(l_1+nl_2)}/k}(V_\Xc^{(1)}(k,kc,z;B))/k$.
    Treating the case $D=D_w$ first, the bound~\eqref{eq:V_nl-bound} results in
    \begin{equation*}
      \begin{aligned}
      \frac{1}{k}V_\Xc^{(1)}(k,c,z;B)
      &\ll \int_{|a|\le B^{1/(l_1+nl_2)}} \frac{B^{1/l_2}}{k\max\{1,|z|\}^{l_1/l_2}\max\{|a|,1\}} \df a \\
      &\ll \frac{B^{1/l_2}\log B}{k\max\{1,|z|\}^{l_1/l_2}}.
      \end{aligned}
    \end{equation*}
    Summing this error over $k\le B^{1/(l_1+nl_2)}$ and $z$ results in a total error of at most
    \begin{equation*}
      \max\{B^{(n+1)/(l_1+nl_2)},B^{1/l_2}\}(\log B)^3 \ll B^{a(\Xc)-\delta(\Xc)}(\log B)^3,
    \end{equation*}
    yielding
    \begin{equation*}
      N_{\Xc,\emptyset}(B) = \sums{1 \le k\le B^{1/(l_1+nl_2)}\\ |z|\le B^{1/(l_1+nl_2)}} \frac{\mu(k)}{k^2} V_\Xc^{(2)}(k,z;B) + O\!\left(B^{a(\Xc)-\delta(\Xc)}(\log B)^3\right).
    \end{equation*}
    For the case $D=D_w+D_z$ (with $z=1$), the estimate~\eqref{eq:V_nl-bound} for $V_\Xc$ together with $\max\{|a|,|c|\} \ge k$ implies
    \begin{equation*}
      \frac{1}{k}V_\Xc^{(1)}(k,c,1;B) \ll \int_{\Rd} \frac{B^{1/l_2}}{k\max\{|a|,|k|\}^{l_1/l_2 +1}} \df a \ll  \frac{B^{1/l_2}}{k^{l_1/l_2+1}}.
    \end{equation*}
    Summing this error over $k \ge 1$ now results in a total error of order of magnitude $B^{1/l_2}$. In other words,
    \begin{equation}\label{eq:eq:N_2-moving-prelim}
      N_{\Xc,\emptyset}(B) = \sums{k\ge 1} \frac{\mu(k)}{k^2} V_\Xc^{(2)}(k,1;B) + O(B^{1/l_2}).
    \end{equation}

    \subsubsection*{Summation over $z$}
    
    This step is only needed for $D=D_w$. Using the symmetry of the volumes in $a$ and $c$, the new volume is at most
    \begin{equation}\label{eq:bound_V2}
        \begin{aligned}
            V_\Xc^{(2)}(k,z;B) &\ll \ints{1\le |c|\le B^{1/(l_1+nl_2)}\\ |a|\le |c|} \frac{B^{1/l_2}}{|c|^{l_1/l_2+1}} \df a\, \df c
            \ll \int_{1\le |c|\le B^{1/(l_1+nl_2)}} \frac{B^{1/l_2}}{|c|^{l_1/l_2}} \df c \\
            &\ll \max\left\{B^{\frac{1}{l_1} + \frac{1-l_1/l_2}{l_1+nl_2}},B^{1/l_2}\right\} \log B
            \ll B^{a(\Xc)-\delta(\Xc)} \log B,
        \end{aligned}
    \end{equation}
    where the factor $\log B$ is only necessary if $l_1=l_2$. Defining
    \begin{equation*}
        V_\Xc^{(3)}(k;B) =
      \int_{|z|\le B^{1/(l_1+nl_2)}} V_\Xc^{(2)}(k,z;B)\df z,
    \end{equation*}
    it follows from~\eqref{eq:bound_V2} that 
    \begin{equation*}
      \left| \frac{1}{k^2}\sum_{|z|\le B^{1/(l_1+nl_2)}} V_\Xc^{(2)}(k,z;B) - \frac{1}{k^2}V_\Xc^{(3)}(k;B)\right|
      \ll \frac{B^{a(\Xc)-\delta(\Xc)} \log B}{k^2},
    \end{equation*}
    and summing $1/k^2$ over $k\ge 1$ contributes only a bounded factor.
    In total,
    \begin{equation*}
        N_{\Xc,\emptyset}(B) = \sum_{k \ge 1}\frac{\mu(k)}{k^2}V_\Xc^{(3)}(k;B) + O\!\left(B^{a(\Xc)-\delta(\Xc)} (\log B)^3\right).
    \end{equation*}

    \subsubsection*{Integrating near the origin}
    The difference between $V_\Xc^{(3)}(k;B)$ and $W_\Xc(B)$ for $D=D_w$ lies in the condition $\max\{|a|,|c|\}\ge k$; removing it introduces the error
    \begin{equation*}
      \ints{|a|,|c|\le k\\|z|\le B^{1/(l_1+nl_2)}}V_\Xc(a,c,z;B).
    \end{equation*}
    Appealing to~\eqref{eq:V_nl-bound} and using its symmetry in $a$ and $c$, this error is bounded by a constant multiple of
    \begin{equation*}
      2\!\ints{|a|\le|c|\le k\\|z|\le B^{1/(l_1+nl_2)}}\!
      \frac{B^{1/l_2}}{\max\{|c|,|z|\}^{l_1/l_2}|c|} \df a \, \df c \, \df z
      \!\ll\!
      \ints{|c|\le k\\|z|\le B^{1/(l_1+nl_2)}}
      \!\frac{B^{1/l_2}}{\max\{|c|,|z|\}^{l_1/l_2}} \df c \, \df z.
    \end{equation*}
    Dividing the integral into ranges with $|z|\le |c|$ and $|c|<|z|$, the one over the former is easily seen to be at most
    \begin{equation}\label{eq:error-remove-max-z<c}
      \int_{|z|\le|c|\le k} \frac{B^{1/l_2}}{|c|^{l_1/l_2}}  \df c \, \df z\ll B^{1/l_2}k^{2-l_1/l_2};
    \end{equation}
    for the latter integral, we first estimate
    \begin{equation*}
      \int_{|c|\le |z|\le  B^{1/(l_1+nl_2)}} \frac{B^{1/l_2}}{|z|^{l_1/l_2}} \df z
      \ll 
      \begin{cases}
        B^{(n+1)/(l_1+nl_2)}, & \text{if $l_1<l_2$,} \\
        B^{1/l_2}(\log B+|\log |c||), & \text{if $l_1=l_2$, and}\\
        B^{1/l_2}|c|^{1-l_1/l_2}, & \text{if $l_2<l_1<2l_2$,}
      \end{cases}
    \end{equation*}
    and thus
    \begin{equation*}
      \ints{0\le |c|\le k\\|c|\le |z|\le  B^{1/(l_1+nl_2)}}   \frac{B^{1/l_2}}{|z|^{l_1/l_2}} \df z\,\df c
      \ll 
      \begin{cases}
        kB^{(n+1)/(l_1+nl_2)}, & \text{if $l_1<l_2$,} \\
        kB^{1/l_2}(\log B+\log k), & \text{if $l_1=l_2$, and}\\
        B^{1/l_2}k^{2-l_1/l_2}, & \text{if $l_2<l_1<2l_2$.}
      \end{cases}
    \end{equation*}
    In any case, both these errors and~\eqref{eq:error-remove-max-z<c} can be summed over $k\le B^{1/(l_1+nl_2)}$, so that
    \begin{equation}\label{eq:W1_bounded_k}
      \left|N_{\Xc,\emptyset}(B)  - \sum_{1\le k \le B^{1/(l_1+nl_2)}} \frac{\mu(k)}{k^2}W_\Xc(B)\right|
      \ll B^{a(\Xc)-\delta(\Xc)} (\log B)^3.
    \end{equation}

    Using the symmetry of \ref{eq:V_nl-bound} in $a,c$, one can bound
    \begin{equation*}
        W_\Xc(B) \ll B^{a(\Xc)}.
    \end{equation*}
    Therefore, extending the sum over $k$ in \ref{eq:W1_bounded_k}
    to all $k\ge 1$ introduces an error of at most
    \begin{align*}
      \sum_{k> B^{1/(l_1+nl_2)}} \frac{1}{k^2}W_\Xc(B) \ll B^{a(\Xc)-\delta(\Xc)}.
    \end{align*}
    This completes the proof in the case $D=D_w$.

    For $D=D_w+D_z$, comparing $V_\Xc^{(2)}(k,1;B)$ and $W_\Xc(B)$, the region of the former integral is described by $H_L(a,b,c,(1+bc)/a,1,1)\le B$ and $\max\{|a|,|c|\}\ge k$, while that of the latter is described by $H_L(a,b,c,bc/a,0,1)\le B$ and $\max\{|a|,|c|\}\ge 1$ (resp. no lower bound on this maximum if $l_1 < l_2$). Replacing $(1+bc)/a$ in the former height condition by $bc/a$ as in the latter introduces an error term of size $O(\log B)$ by the same arguments used to obtain the estimate~\eqref{eq:compare-V-V'} and its sum over the remaining variables in Lemma~\ref{lem:sum_V'}. Replacing $z=1$ by $z=0$ preserves the value of the integral as $\max\{|a|,|c|\}\ge k\ge 1$, so that $z$ never contributes to the maxima anyway. Removing the lower bound $\max\{|a|,|c|\}\ge k$ if $l_1<l_2$ introduces an error of at most
    \begin{equation*}
    2 \int_{|a|\le|c|< k} \frac{B^{1/l_2}}{|c|^{l_1/l_2+1}}\df a\,\df c \ll \int_{|c|\le k} \frac{B^{1/l_2}}{|c|^{l_1/l_2}}\df c \ll  B^{1/l_2}k^{1-l_1/l_2},
    \end{equation*}
    whose sum over $k\ge 1$ converges to a function in $O(B^{1/l_2})$ after multiplying it with the factor $k^{-2}$ in~\eqref{eq:eq:N_2-moving-prelim}. If $l_1=l_2$, replacing $\max\{|a|,|c|\}\ge k$ by $\max\{|a|,|c|\}\ge 1$ results in the error
    \begin{equation*}
      2 \ints{1\le|c|\le k\\|a|\le |c|} \frac{B^{1/l_2}}{|c|^{2}}\df a\,\df c \ll \int_{1 \le |c|\le k} \frac{B^{1/l_2}}{|c|}\df c \ll B^{1/l_2} \log k,
    \end{equation*}
    whose sum over $k\ge 1$ converges after multiplication by $k^{-2}$ as well.
\end{proof}

\begin{lemma}
 The volume $W_\Xc(B)$ in Proposition~\ref{prop:rigid_integrals} can be expressed as
    \begin{align*}
        &\frac{B^{a(\Xc)}}{2} \int_{H_L(a,b,c,(bc+z^{n+1})/a,z,1) \le 1} \frac{\df a\,\df b\,\df c\,\df z}{|a|}
        && \text{if $D=D_w$,}\\[.5em]
        &B^{a(\Xc)} \int_{H_L(a,b,c,bc/a,0,1) \le 1} \frac{\df a\,\df b\,\df c}{|a|}
        &&\text{if $D=D_w+D_z$, $l_1<l_2$, and}\\[.5em]
        &\frac{2B^{a(\Xc)} \log B}{l_1(n+1)} \int_{H_L(1,b,c,bc,0,0) \le 1} \df b\,\df c + O(B^{a(\Xc)})
        && \text{if $D=D_w+D_z$, $l_1=l_2$.}
    \end{align*}
\end{lemma}

\begin{proof}
    For $D=D_w$, the change of variables $(a,c,z) \mapsto B^{-1/(l_1+nl_2)}(a,c,z)$ and $b\mapsto B^{-n/(l_1+nl_2)}b$ transforms $W_\Xc(B)$ to the integral on the right. If $D=D_w+D_z$ and $l_1<l_2$, the same change of variables without $z$ achieves this end.
    In the final case, with $D=D_w+D_z$ and $l_1=l_2$, we first replace $\max\{|a|,|c|\} \ge 1$ by $|a| \ge 1$. Using \eqref{eq:constants-moving-explicit}, the difference is
    \begin{equation*}
        \ints{H_L(a,b,c,bc/a,0,1) \le B\\|a|\le 1,\ |c| \ge 1} \frac{\df a\,\df b\,\df c}{|a|} \ll \int_{|a|\le 1,\ |c| \ge 1} \frac{B^{1/l_2}}{|c^2|} \df a\,\df c \ll B^{a(\Xc)}.
    \end{equation*}
    Observe that 
    \begin{equation*}
        H_L(a,b,c,bc/a,0,1) = \max\{H_L(a,b,c,bc/a,0,0),|a|^{l_1(n+1)},|c|^{l_1(n+1)}\};
    \end{equation*}
    removing $|c|^{l_1(n+1)} \le B$ from this height condition once more results in an error of
    \begin{equation*}
        \ints{H_L(a,b,c,bc/a,0,1) \le B\\|a|\ge 1,\ |c|^{l_1(n+1)}>B} \frac{\df a\,\df b\,\df c}{|a|} \ll \ints{1 \le |a| \le B^{1/(l_1(n+1))}\\|c| > B^{1/(l_1(n+1))}} \frac{B^{1/l_2}}{|c^2|} \df a\, \df c\ll B^{a(\Xc)}.
    \end{equation*}
    The resulting main term is
    \begin{equation*}
        \int_{1 \le |a| \le B^{1/(l_1(n+1))}} \frac{1}{|a|} \left(\int_{H_L(a,b,c,bc/a,0,0) \le B} \df b\, \df c\right) \df a;
    \end{equation*}
    via the change of variables $b\mapsto B^{-1/l_1}ab, c\mapsto a^{-1}c$, its inner integral is equal to
    \begin{equation*}
        B^{a(\Xc)} \int_{H_L(1,b,c,bc,0,0) \le 1} \df b\, \df c
    \end{equation*}
    and in particular independent of $a$. Integrating $|a|^{-1}$ over $a$ results in the remaining factor $2 (l_1(n+1))^{-1} \log B$.
\end{proof}

\begin{lemma}\label{lem:W_interpretation}
    The volumes in Proposition~\ref{prop:rigid_integrals} satisfy
    \begin{equation*}
        W_\Xc(B) = \frac{\alpha(\Xc)\omega_{\Xc,\infty}}{a(\Xc)}B^{a(\Xc)}(\log B)^{b(\Xc)-1}
    \end{equation*}
    if $D=D_w$ or $l_1<l_2$, and the same expression holds true up to the addition of an error term $O(B^{a(\Xc)})$ if $D=D_w+D_z$ and $l_1=l_2$.
\end{lemma}

\begin{proof}
    For $D=D_w$, our starting point is \eqref{eq:omega_infty_Dw}.
    We observe that
    \begin{equation*}
        H_L(1,1,c,c+z^{n+1}w,z,w) = |b|^{-l_2} H_L(1,b,c,bc+z^{n+1},z,1)
    \end{equation*}
    for $w=b^{-1}$ using the homogeneity of $H_L$. By this change of variables and using the identity $1+e_\Xc=l_2 a(\Xc)$ (in the second coordinate of \eqref{eq:def-adjoint}) in the exponent of $b$, we obtain
    \begin{equation*}
        \omega_{\Xc,\infty} = \int \frac{\df b\,\df c\,\df z}{H_L(1,b,c,bc+z^{n+1},z,1)^{a(\Xc)}}.
    \end{equation*}
    Using the identity
    \begin{equation}\label{eq:identity_1/H}
        \frac{1}{H} = \beta \int_{0 \le t^\beta H \le 1} t^{\beta-1} \df t
    \end{equation}
    for $\beta=n+2$, the equality
    \begin{equation*}
        H_L(t,t^nb,tc,t^n(bc+z^n),tz,1) = t^{l_1+nl_2}H_L(1,b,c,bc+z^{n+1},z,1)
    \end{equation*}
    following from the homogeneity of the height function, and $n+2=(l_1+nl_2)a(\Xc)$ in the exponent of $t$, this density can be rewritten as
    \begin{equation*}
        \omega_{\Xc,\infty} = (n+2)\ints{H_L(t,t^nb,tc,t^n(bc+z^n),tz,1) \le 1\\t > 0} t^{n+1} \df t\, \df b\, \df c\, \df z.
    \end{equation*}
    Next, the change of variables $(a',b',c',z')=(t,t^nb,tc,tz)$ for positive $a'$ and extending the resulting integral to negative $a'$ (resulting in a factor $1/2$) yields
    \begin{equation*}
        \omega_{\Xc,\infty} = \frac{n+2}{2} \ints{H_L(a',b',c',(b'c'+z'^{n+1})/a',z',1) \le 1} \frac{\df a'\,\df b'\,\df c'\,\df z'}{|a'|},
    \end{equation*}
    and the statement follows using the description~\eqref{eq:alpha_not_loganticanonical} of $\alpha(\Xc)$ and its consequence $n+2 = a(\Xc)/\alpha(\Xc)$.

    For $D=D_w+D_z$ and $l_1<l_2$, we argue similarly, except that we have $z=0$ and use the case $\beta = n+1$ of \eqref{eq:identity_1/H}, and $n+1 = a(\Xc)/\alpha(\Xc)$.
    If finally $D=D_w+D_z$ and $l_1=l_2$, we start with~\eqref{eq:omega_infty_DwDz_2}.
    Applying \eqref{eq:identity_1/H} for $\beta=1$, $t=b$, the identity
    \begin{equation*}
        |b|^{l_2} H_L(1,1,c,c,0,0) = H_L(1,b,c,bc,0,0)
    \end{equation*}
    following from the homogeneity of $H_L$, and symmetrically extending the integral to negative $b$ (resulting in a factor $1/2$), this density is
    \begin{equation*}
        \omega_{\Xc,\infty} = 2 \int_{H_L(1,b,c,bc,0,0) \le 1} \df b\, \df c.
    \end{equation*}
    Finally, observe that $1/(l_1(n+1)) = \alpha(\Xc)/a(\Xc)$ using \eqref{eq:alpha_loganticanonical}.
\end{proof}

\subsection{The asymptotic formula}
Comparing this counting result with the analysis in Section~\ref{sec:leading-constants} leads to Theorem~\ref{thm:main-thm-abstract} in the following concrete form.
\begin{prop}
  If $\Xc = (\Xm_n,D,L)$ is adjoint rigid, then
  \begin{equation*}
    N_{\Xc,\emptyset}(B) = \frac{\alpha(\Xc)}{a(\Xc)}
    \prod_{v\in \Omega_\Qd} \omega_{\Xc,v}
    B^{a(\Xc)}(\log B)^{b(\Xc)-1}(1+O(1/\log B));
  \end{equation*}
  if it is not, then
  \begin{equation*}
    N_{\Xc,\emptyset}(B) = \sum_{t\in \Pd^2(\Qd)}
    \frac{\alpha(\Xc_t)}{a(\Xc)}
    \prod_{v\in \Omega_\Qd} \omega_{\Xc_t,v}
    B^{a(\Xc)}(\log B)^{b(\Xc)-1}(1+O(1/\log B)).
  \end{equation*}
\end{prop}
\begin{proof}
  In the rigid case, this follows from Proposition~\ref{prop:rigid_integrals}, Lemma~\ref{lem:W_interpretation}, and \eqref{eq:omega_p}. The only remaining case is $l_1=2l_2$. In this case, $L$ is a multiple of the log anticanonical bundle and $(X, D)$ is a partial equivariant compactification of $\mathds G_\mathrm{a}^3$; thus, the statement is a special case of~\cite{CLT12}.

  In the moving case, using the notation from \eqref{eq:counting_function}, $N_{\Xc_t,\emptyset}(B)$ denotes the number of integral points in the fiber $\phi^{-1}(t)$ of $L$-height at most $B$. On the one hand, recall that
  \begin{equation*}
      N_{\Xc_t,\emptyset}(B) = c'(\Xc_t)B^{1/l_2} + O_{t}(1)
  \end{equation*}
  by \eqref{eq:points_on_fiber}. On the other hand,
  \begin{equation*}
      N_{\Xc_t,\emptyset}(B) = c(\Xc_t)
    B^{a(\Xc_t)} + O_{t}(B^{a(\Xc_t)-\delta(t)})
  \end{equation*}
  with
  \begin{equation*}
    c(\Xc_t) = \frac{\alpha(\Xc_t)}{a(\Xc_t)}
    \prod_{v\in \Omega_\Qd} \omega_{\Xc_t,v}
  \end{equation*}
  for some $\delta(t)>0$ by \cite{CLT12}. Therefore, the leading constants and exponents of $B$ have to agree. Now the result follows from Proposition~\ref{prop:count-moving}.
\end{proof}

\bibliographystyle{amsalpha}

\end{document}